\documentclass[12pt]{amsart}

\usepackage{amsmath,amsthm,amssymb,graphicx,tikz,setspace}
\usepackage[margin=1in]{geometry}
\usepackage[hidelinks]{hyperref}

\newtheorem{theorem}{Theorem}[section]
\newtheorem{lemma}[theorem]{Lemma}
\newtheorem{proposition}[theorem]{Proposition}
\newtheorem{corollary}[theorem]{Corollary}

\theoremstyle{definition} 
\newtheorem{definition}[theorem]{Definition}

\newtheorem{remark}[theorem]{Remark}

\DeclareMathOperator{\Cone}{Cone}
\DeclareMathOperator{\Hom}{Hom}
\DeclareMathOperator{\id}{id}

\newcommand{\A}{\mathcal{A}}
\newcommand{\Cc}{\mathcal{C}}
\newcommand{\Dc}{\mathcal{D}}
\newcommand{\Ec}{\mathcal{E}}
\newcommand{\F}{\mathbb{F}}
\newcommand{\gl}{\mathfrak{gl}}
\newcommand{\Modl}[1]{#1\textrm{-Mod}}
\newcommand{\Sc}{\mathcal{S}}
\newcommand{\Uc}{\mathcal{U}}
\newcommand{\Z}{\mathbb{Z}}
\newcommand{\Zc}{\mathcal{Z}}

\newcommand{\ootimes}{ 
  \mathbin{
    \mathchoice
      {\buildcircleotimes{\displaystyle}}
      {\buildcircleotimes{\textstyle}}
      {\buildcircleotimes{\scriptstyle}}
      {\buildcircleotimes{\scriptscriptstyle}}
  } 
}

\newcommand\buildcircleotimes[1]{%
  \begin{tikzpicture}[baseline=(X.base), inner sep=0, outer sep=0]
    \node[draw,circle] (X)  {$#1\otimes$};
  \end{tikzpicture}%
}

\title[Decategorified higher actions]{On the decategorification of some higher actions in Heegaard Floer homology}
\author[Andrew Manion]{Andrew Manion}
\address{Department of Mathematics, North Carolina State University, 2108 SAS Hall, Raleigh, NC 27695}
\email{ajmanion@ncsu.edu}

\begin{document}

\begin{abstract}
    We decategorify the higher actions on bordered Heegaard Floer strands algebras from recent work of Rouquier and the author and identify the decategorifications with certain actions on exterior powers of homology groups of surfaces. We also suggest an interpretation for these actions in the language of open-closed TQFT, and we prove a corresponding gluing formula.
\end{abstract}

\maketitle

\section{Introduction}

In \cite{ManionRouquier}, Rapha{\"e}l Rouquier and the author define a tensor product operation for higher representations of the dg monoidal category from \cite{KhovOneHalf}, which we call $\Uc$, and use it to reformulate aspects of cornered Heegaard Floer homology \cite{DM,DLM}. Part of this work involves defining 2-actions of $\Uc$ on the dg algebras $\A(\Zc)$ that bordered Heegaard Floer homology assigns to combinatorial representations $\Zc$ of surfaces.

Ignoring gradings and thus working with decategorifications over $\F_2$, one can view $\Uc$ as a categorification of the algebra $\F_2[E]/(E^2)$ (an $\F_2$ analogue of $U(\gl(1|1)^+)$), while if $\Zc$ is a representation of a surface $F$, then $\A(\Zc)$ categorifies the vector space $\wedge^* H_1(F,S_+; \F_2)$ where $S_+$ is a distinguished subset of the boundary of $F$. Thus, the 2-actions from \cite{ManionRouquier} should categorify actions of $\F_2[E]/(E^2)$ on $\wedge^* H_1(F,S_+; \F_2)$; the goal of this paper is to identify these actions explicitly using certain topological operations and to give an interpretation of these actions in the setting of open-closed TQFT.

To make things more precise, we recall that following Zarev \cite{Zarev} (but generalizing his definition slightly), a sutured surface is $(F,S_+,S_-,\Lambda)$ where $F$ is a compact oriented surface and $\Lambda$ is a finite set of points in $\partial F$ dividing $\partial F$ into alternating subsets $S_+$ and $S_-$. We impose no topological restrictions, but note that the sutured surfaces representable by arc diagrams $\Zc$ are those such that in each connected component of $F$ (not of $\partial F$), both $S_+$ and $S_-$ are nonempty (unlike Zarev \cite{Zarev}, we allow arc diagrams to have circle components as well as interval components, and we do not impose non-degeneracy). In particular, no closed surface can be represented by an arc diagram.

For an arc diagram $\Zc$ representing a sutured surface $(F,S_+,S_-,\Lambda)$, and each interval component $I$ of $S_+$, the constructions of \cite{ManionRouquier} define a 2-action of $\Uc$ on $\A(\Zc)$. On the other hand, there is a map $\phi_I$ from $H_1(F,S_+;\F_2)$ to $\F_2$ taking an element of $H_1(F,S_+;\F_2)$ to its boundary in $H_0(S_+;\F_2)$ and then pairing with the cohomology class of $I$. By summing $\phi_I$ over tensor factors, for $k \geq 1$ we get a map from $T^k H_1(F,S_+;\F_2)$ to $T^{k-1} H_1(F,S_+;\F_2)$ which induces a map $\Phi_I$ from $\wedge^k H_1(F,S_+;\F_2)$ to $\wedge^{k-1} H_1(F,S_+;\F_2)$.

\begin{theorem}\label{thm:IntroDecategorifiedUActions}
The 2-action of $\Uc$ on $\A(\Zc)$ corresponding to $I$ categorifies the action of $\F_2[E]/(E^2)$ on $\wedge^* H_1(F,S_+;\F_2)$ in which $E$ acts by $\Phi_I$.
\end{theorem}

See Theorem~\ref{thm:MainDecatResult} below for a more detailed statement of Theorem~\ref{thm:IntroDecategorifiedUActions}.

\subsection*{A TQFT interpretation}

It is natural to ask whether the actions of $\F_2[E]/(E^2)$ on $\wedge^* H_1(F,S_+;\F_2)$ fit into a TQFT framework, with associated gluing results. Indeed, \cite{ManionRouquier} reformulates and strengthens Douglas--Manolescu's gluing theorem for the algebras $\A(\Zc)$, which applies for certain decompositions of surfaces along 1-manifolds (given by certain decompositions of the arc diagram $\Zc$). One could hope that such gluing theorems exist in even greater generality for the decategorified surface invariants $\wedge^* H_1(F,S_+;\F_2)$, yielding a TQFT-like construction for 1- and 2-manifolds.

\begin{remark}
Heegaard Floer homology is, in some non-axiomatic sense, a 4-dimensional TQFT (spacetimes are 4-dimensional); accordingly, its decategorification should be a type of 3-dimensional TQFT involving the vector spaces $\wedge^* H_1(F,S_+;\F_2)$ (and e.g. the Alexander polynomials of knots). The constructions under consideration for 1- and 2-manifolds should be part of a (loosely defined) extended-TQFT structure for decategorified Heegaard Floer homology.
\end{remark}

A first observation is that a sutured surface $(F,S_+,S_-,\Lambda)$ is nearly the same data as a morphism in the 2-dimensional open-closed cobordism category. As described in \cite{LaudaPfeiffer}, the objects of this category are finite disjoint unions of oriented intervals and circles. For two such objects $X,Y$, a morphism from $X$ to $Y$ is a compact oriented surface with its boundary decomposed into black regions (identified with $X \sqcup Y$) and colored regions. If $(F,S_+,S_-,\Lambda)$ is a sutured surface and we label each component of $S_+$ as ``incoming'' or ``outgoing,'' we get a morphism from $S_+^{\mathrm{in}}$ to $S_+^{\mathrm{out}}$ in this cobordism category. The black part of the boundary is $S_+$ and the colored part is $S_-$.

The actions of $\F_2[E]/(E^2)$ on $\wedge^* H_1(F,S_+;\F_2)$ suggest that one could try to assign the category of finite-dimensional $\F_2[E]/(E^2)$-modules to an interval. A sutured surface, with its $S_+$ boundary components labeled as incoming or outgoing, would be assigned a bimodule over tensor powers of $\F_2[E]/(E^2)$. For simplicity, we will restrict our attention here to sutured surfaces with no circular $S_+$ boundary components (all components of $S_+$ are intervals). 

For a surface $F_1$ with $m$ intervals in its outgoing boundary and another surface $F_2$ with $m$ intervals in its incoming boundary, let $F = F_2 \cup_{[0,1]^m} F_1$. We would want the bimodule of $F$ to be a tensor product over $(\F_2[E]/(E^2))^{\otimes m}$ of the bimodules assigned to $F_1$ and $F_2$. The next theorem says this is true up to isomorphism.

\begin{theorem}\label{thm:IntroTQFT}
For $F_1$, $F_2$, and $F$ as above, we have
\[
\wedge^* H_1(F,S_+;\F_2) \cong \wedge^* H_1(F_2,S_+;\F_2) \otimes_{(\F_2[E]/(E^2))^{\otimes m}} \wedge^* H_1(F_1,S_+;\F_2).
\]
Thus, the vector spaces $\wedge^* H_1(F,S_+;\F_2)$ give a functor from the ``open sector'' of the open-closed cobordism category into (algebras, bimodules up to isomorphism).
\end{theorem}

In fact, a slightly more general version of Theorem~\ref{thm:IntroTQFT} holds in which $F_1$ and $F_2$ can have $S_+$ circles in their boundaries as long as we are not gluing along them; see Theorem~\ref{thm:TQFTGluing} below.

\subsection*{The tensor product case}

As a special case of Theorem~\ref{thm:IntroTQFT}, we can glue interval $S_+$ components of two surfaces $F', F''$ to the two input intervals of the ``open pair of pants'' cobordism shown in Figure~\ref{fig:OpenPairOfPants}. Let $P = F_1$ be the open pair of pants, let $F_2 = F' \sqcup F''$, and let $F$ be the glued surface. We can identify $\wedge^* H_1(P,S_+;\F_2)$ with $(\F_2[E]/(E^2))^{\otimes 2}$, with right action of $(\F_2[E]/(E^2))^{\otimes 2}$ given by multiplication and left action of $\F_2[E] / (E^2)$ given by the coproduct
\[
\Delta(E) = E \otimes 1 + 1 \otimes E
\]
(in fact, $\F_2[E] / (E^2)$ is a Hopf algebra with this coproduct together with counit $\varepsilon(E) = 0$ and antipode $S(E) = E$).

\begin{figure}
    \centering
    \includegraphics[scale=0.7]{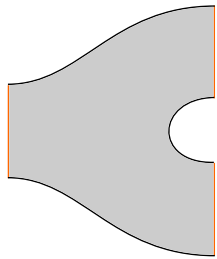}
    \caption{The open pair of pants; the $S_+$ boundary is shown in orange and the $S_-$ boundary is shown in black (loosely following the visual conventions of \cite{Zarev}). Specifically, the input $S_+$ boundary is on the right while the output $S_+$ boundary is on the left.}
    \label{fig:OpenPairOfPants}
\end{figure}

\begin{corollary}\label{cor:IntroTensor}
We have 
\[
\wedge^* H_1(F,S_+;\F_2) \cong \wedge^* H_1(F',S_+;\F_2) \otimes \wedge^* H_1(F'',S_+;\F_2) 
\]
where the tensor product $\otimes$ is taken in the tensor category of finite-dimensional modules over the Hopf algebra $\F_2[E]/(E^2)$.
\end{corollary}

We can view Corollary~\ref{cor:IntroTensor} as a decategorification of the gluing result from \cite{ManionRouquier} based on the higher tensor product operation $\ootimes$. Thus, Theorem~\ref{thm:IntroTQFT} suggests (at least at the decategorified level) a more general TQFT framework for the $\ootimes$-based gluing results of \cite{ManionRouquier}.

\subsection*{Relationship to other work}

Probably the closest analogue to the structures considered here can be found in Honda--Kazez--Mati{\'c}'s paper \cite{HKMTQFT}. The data of a sutured surface $(F,S_+,S_-,\Lambda)$ as discussed here is equivalent to the data $(\Sigma,F)$ considered in \cite[Section 7.1]{HKMTQFT} (our $F$ is Honda--Kazez--Mati{\'c}'s $\Sigma$ and our $\Lambda$ is their $F$). The vector space $\wedge^* H_1(F,S_+;\F_2)$ is isomorphic to an $\F_2$ version of Honda--Kazez--Mati{\'c}'s $V(\Sigma, F)$ which was subsequently studied by Mathews \cite{Mathews1, Mathews2, Mathews3, Mathews4, MathewsSchoenfeld}. In our notation, Honda--Kazez--Mati{\'c} view this vector space as the sutured Floer homology of $F \times S^1$ with sutures given by $\Lambda \times S^1$, rather than as a Grothendieck group associated to $\A(\Zc)$. In other words, their surface invariants come from ``trace decategorification'' of 3-dimensional Heegaard Floer invariants rather than from Grothendieck-group-based decategorification of 2-dimensional Heegaard Floer invariants; these notions often agree, as they do here. See Cooper \cite{Cooper} for related work in the contact setting that discusses vector spaces similar to $\wedge^* H_1(F,S_+;\F_2)$ in relation to Grothendieck groups of formal contact categories.

We can think of the gluings in Theorem~\ref{thm:IntroTQFT} as successive self-gluings of two $S_+$ intervals in a sutured surface. These gluings can be interpreted as special cases of Honda--Kazez--Mati{\'c}'s gluings, where their gluing subsets $\gamma, \gamma'$ cover our gluing $S_+$ intervals and extend a small bit past them on both sides. However, Honda--Kazez--Mati{\'c} only assert the existence of a gluing map from the vector space of the original surface to the vector space of the glued surface (satisfying certain properties). Theorem~\ref{thm:IntroTQFT} goes farther for the special gluings under consideration in that it shows how the vector space of the larger surface is recovered up to isomorphism as a tensor product.

Integral versions of the vector spaces $\wedge^*(F,S_+;\F_2)$, especially for closed $F$ or $F$ with one boundary component (and implicitly $|\Lambda| = 2$), have also been studied in the context of TQFT invariants for 3-manifolds starting with Frohman and Nicas in \cite{FrohmanNicas} (see also \cite{DonaldsonTQFT, KerlerHomologyTQFT}). Building on work of Petkova \cite{PetkovaDecat}, Hom--Lidman--Watson show in \cite{HLW} that bordered Heegaard Floer homology (in the original formulation of \cite{LOTBorderedOrig} where $F$ is closed) can be viewed as categorifying the 2+1 TQFT described in \cite{DonaldsonTQFT} in which a surface $F$ is assigned $\wedge^* H_1 (F)$. Our perspective here differs in that we follow Zarev \cite{Zarev} rather than \cite{LOTBorderedOrig} and in that instead of 2+1 TQFT structure we are (loosely) looking at the lower two levels of a 1+1+1 TQFT.

Finally, the fact that the topological gluing considered in \cite{ManionRouquier} can be viewed as the above open-pair-of-pants gluing was already noted in \cite[Section 7.2.5]{ManionRouquier}, which also contains speculations about the connection to open-closed TQFT and extended TQFT. 

\subsection*{Future directions}

It would be desirable to treat 1-, 2-, and 3-manifolds at the same time, integrating the gluing results for surfaces here with the 3-manifold invariants mentioned above in something like a 1+1+1 TQFT. One obstacle to doing this appears to be that the isomorphism in the statement of Theorem~\ref{thm:IntroTQFT} is not canonical and depends on suitable choices of bases. Given arbitrary elements of $\wedge^* H_1(F_1,S_+;\F_2)$ and $\wedge^* H_1(F_2,S_+;\F_2)$, it is not clear how to pair them to get an element of $\wedge^* H_1(F,S_+;\F_2)$ in a canonical way. 

It would also be desirable to categorify Theorem~\ref{thm:IntroTQFT}, such that the $\ootimes$-based gluing results of \cite{ManionRouquier} are recovered by gluing with an open pair of pants as in Corollary~\ref{cor:IntroTensor}. Just as the proof of Theorem~\ref{thm:IntroTQFT} depends on a choice of basis, it seems likely that a categorification of this theorem will depend on the arc diagrams $\Zc$ chosen to represent the surfaces. For general arc diagrams $\Zc_1$ and $\Zc_2$ representing the surfaces $F_1$ and $F_2$ of Theorem~\ref{thm:IntroTQFT}, it is not even clear how one should glue these diagrams to get an arc diagram for the glued surface $F$ (speculatively, something like \cite[Figure 5(b)]{KaufmannPenner} followed by an ``unzip'' operation may be relevant). 

Finally, preliminary computations indicate that close relatives of $\wedge^* H_1(F,S_+;\F_2)$ should arise in a TQFT with better structural properties than the ``open'' TQFT considered here, specifically one that is extended down to points and defined at least for all 0-, 1-, and 2-manifolds, with appropriate gluing theorems (including for gluing along circles). In work in progress, we study this extended TQFT as well as its relationship to the constructions of this paper.

\subsection*{Organization}

In Section~\ref{sec:U} through \ref{sec:HigherActions} we review $\Uc$, the algebras $\A(\Zc)$, and the higher actions from \cite{ManionRouquier}. Section~\ref{sec:Decategorification} discusses decategorification for $\Uc$ and $\A(\Zc)$, showing that in the sense considered here, $\A(\Zc)$ categorifies $\wedge^* H_1(F,S_+;\F_2)$. Section~\ref{sec:HomologyActions} decategorifies the 2-actions of $\Uc$ on $\A(\Zc)$ from \cite{ManionRouquier} and proves Theorem~\ref{thm:IntroDecategorifiedUActions}. Section~\ref{sec:TQFT} proves a generalized version of Theorem~\ref{thm:IntroTQFT}, and Section~\ref{sec:Tensor} discusses Corollary~\ref{cor:IntroTensor} in more generality.

\subsection*{Acknowledgments}

We would like to thank Bojko Bakalov, Corey Jones, Robert Lipshitz, and Rapha{\"e}l Rouquier for useful conversations. This research was supported by NSF grant number DMS-2151786.

\section{Decategorifying higher actions on strands algebras}

\subsection{The dg monoidal category \texorpdfstring{$\Uc$}{U}}\label{sec:U}

The following definition originated in \cite{KhovOneHalf} and was partly inspired by the strands dg algebras $\A(\Zc)$ in Heegaard Floer homology (we review these in Section~\ref{sec:StrandsAlgs}). While Khovanov works over $\Z$, we work over $\F_2$ instead to interact properly with the $\F_2$-algebras $\A(\Zc)$.

\begin{definition}
Let $\Uc$ denote the strict $\F_2$-linear dg monoidal category freely generated (under $\otimes$ and composition) by an object $e$ and an endomorphism $\tau$ of $e \otimes e$ modulo the relations $\tau^2 = 0$ and
\[
(\id_e \otimes \tau) \circ (\tau \otimes \id_e) \circ (\id_e \otimes \tau) = (\tau \otimes \id_e) \circ (\id_e \otimes \tau) \circ (\tau \otimes \id_e).
\]
We set $d(\tau) = 1$, and we let $\tau$ have degree $-1$ (we use the convention that differentials increase degree by $1$).
\end{definition}

The endomorphism algebra of $e^{\otimes n} \in \Uc$ is the dg algebra referred to as $H_n^-$ in \cite{KhovOneHalf} (tensored with $\F_2$); in the language used in \cite{ManionRouquier} it is a nilHecke algebra with a differential and in the language used in \cite{DM} it is a nilCoxeter algebra. We will use $NC_n$ to denote the $\F_2$ version of this algebra. It has a graphical interpretation: $\F_2$-basis elements of $NC_n$ are pictures like Figure~\ref{fig:NCn_basis}, with $n$ strands going from bottom to top (these pictures are in bijection with permutations on $n$ letters). Multiplication is defined by vertical concatenation, with $ab$ obtained by drawing $a$ below $b$, except that if two strands cross and then uncross in the stacked picture (i.e. if the stacked picture has a double crossing) then the product is defined to be zero. The differential is defined by summing over all ways to resolve a crossing (see Figure~\ref{fig:CrossingResolution}), except that if a crossing resolution produces a double crossing between two strands then it contributes zero to the differential (see Figure~\ref{fig:DisallowedResolution}). The endomorphism $\tau$ of $e \otimes e$ is represented by a single crossing between two strands.

\begin{figure}
    \centering
    \includegraphics[scale=0.7]{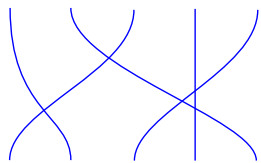}
    \caption{A basis element of $NC_n$ for $n=5$.}
    \label{fig:NCn_basis}
\end{figure}

\begin{figure}
    \centering
    \includegraphics[scale=0.7]{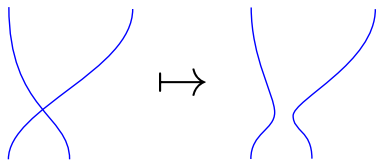}
    \caption{Resolving a crossing.}
    \label{fig:CrossingResolution}
\end{figure}

\begin{figure}
    \centering
    \includegraphics[scale=0.7]{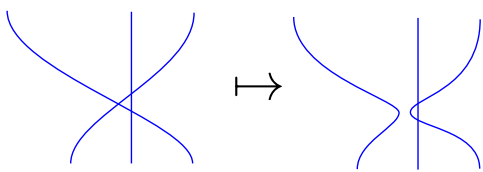}
    \caption{A resolution that produces a double crossing and thus does not contribute to the differential on $NC_n$.}
    \label{fig:DisallowedResolution}
\end{figure}

\subsection{Strands algebras}\label{sec:StrandsAlgs}

Let $\Zc$ be an arc diagram as in \cite[Definition 2.1.1]{Zarev}, except that we allow (oriented) circles as well as intervals in $\mathbf{Z}$, and we do not impose any non-degeneracy condition. Thus, $\Zc$ consists of:
\begin{itemize}
    \item a finite collection $\mathbf{Z} = \{Z_1, \ldots, Z_l\}$ of oriented intervals and circles;
    \item a finite set of points $\mathbf{a}$ (with $|\mathbf{a}|$ even) in the interiors of the $Z_i$ for $1 \leq i \leq l$;
    \item a 2-1 matching $M$ of the points in $\mathbf{a}$.
\end{itemize}
An example is shown in Figure~\ref{fig:ArcDiagram}.

\begin{figure}
    \centering
    \includegraphics[scale=0.7]{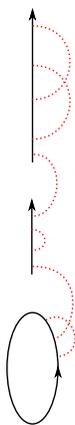}
    \caption{An arc diagram $\Zc = (\mathbf{Z},\mathbf{a}, M)$; $\mathbf{Z}$ consists of two intervals and a circle, $\mathbf{a}$ is the set of endpoints of the red (dotted) arcs, and $M$ matches the two endpoints of each red arc.}
    \label{fig:ArcDiagram}
\end{figure}

The definition of the dg strands algebra $\A(\Zc)$ over $\F_2$, from \cite[Definition 2.2.2]{Zarev}, generalizes in a straightforward way to this setting and is a special case of the general strands algebras treated in detail in \cite{ManionRouquier}. One can view $\A(\Zc)$ as being defined by specifying an $\F_2$ basis consisting of certain pictures, along with rules for multiplying and differentiating basis elements. 

\begin{figure}
    \centering
    \includegraphics[scale=0.7]{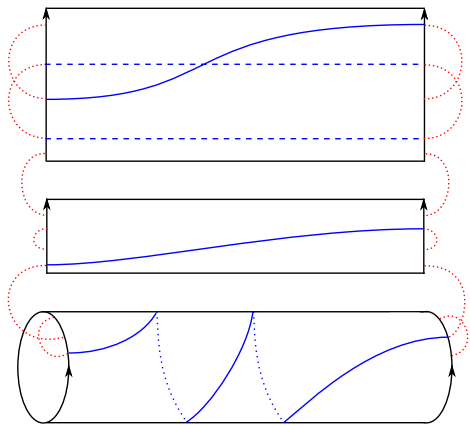}
    \caption{A strands picture (basis element for $\A(\Zc)$).}
    \label{fig:StrandsPicture}
\end{figure}

\begin{definition}\label{def:StrandsPicture}
A \emph{strands picture} is a collection of strands drawn in $[0,1] \times \mathbf{Z}$, each with its left endpoint in $\{0\} \times \mathbf{a}$ and its right endpoint in $\{1\} \times \mathbf{a}$. The strands can be either solid or dotted and are considered only up to homotopy relative to the endpoints; by convention, strands are drawn ``taut,'' sometimes with a bit of curvature for visual effect (see Figure~\ref{fig:StrandsPicture}). They must satisfy the following rules:
\begin{itemize}
    \item Strands cannot move against the orientation of $\mathbf{Z}$ when moving from left to right (from $0$ to $1$ in $[0,1]$).
    \item No solid strands are horizontal, while all dotted strands are horizontal.
    \item If a solid strand has its left endpoint at $a \in \mathbf{a}$, and $a$ is matched to $a' \in \mathbf{a}$ under $M$, then no strand can have its left endpoint at $a'$, and similarly for right endpoints.
    \item If a dotted strand has its left (and thus right) endpoint at $a \in \mathbf{a}$, and $a$ is matched to $a' \in \mathbf{a}$ under $M$, then there must be another dotted strand with its left (and thus right) endpoint at $a'$ (we say this dotted strand is matched with the first one).
\end{itemize}
\end{definition}

\begin{figure}
    \centering
    \includegraphics[scale=0.7]{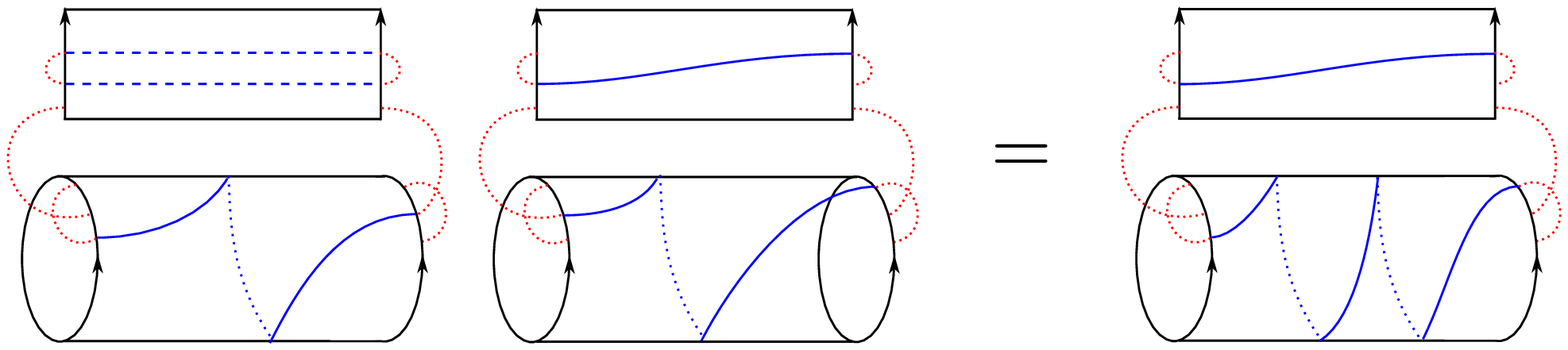}
    \caption{Example of a product in $\A(\Zc)$.}
    \label{fig:StrandsProduct}
\end{figure}

\begin{definition}\label{def:StrandsAlgebra}
As a $\F_2$-vector space, $\A(\Zc)$ is defined to be the formal span of such strands pictures, so that strands pictures form an $\F_2$ basis for $\A(\Zc)$. The product of two basis elements of $\A(\Zc)$ is defined by concatenation (see Figure~\ref{fig:StrandsProduct}), with the following subtleties:
\begin{itemize}
    \item If some solid strand has no strand to concatenate with, or if in some matched pair of dotted strands $\{s,s'\}$, neither $s$ nor $s'$ has a strand to concatenate with, the product is zero.
    \item When concatenating a solid strand with a dotted strand, one erases the dotted strand matched to the one involved in the concatenation, and makes the concatenated strand solid.
    \item If a double crossing is formed upon concatenation, the product of the basis elements is defined to be zero.
\end{itemize}

The differential of a basis element of $\A(\Zc)$ is the sum of all strands pictures formed by resolving a crossing in the original strands picture (in the sense of Figure~\ref{fig:CrossingResolution} above), with the following subtleties:
\begin{itemize}
    \item When resolving a crossing between a solid strand and a dotted strand, one erases the dotted strand matched to the one involved in the crossing resolution, and makes both the resolved strands solid.
    \item If a double crossing is formed upon resolving a crossing (as in Figure~\ref{fig:DisallowedResolution} above), then this crossing resolution does not contribute a term to the differential.
\end{itemize}
\end{definition}

\begin{remark}
While $\Uc$ is a dg category and not just a differential category (its morphism spaces are graded), the grading on $\A(\Zc)$ is much more complicated: it is a grading by a nonabelian group $G(\Zc)$ rather than by $\Zc$, and it depends on a choice of ``grading refinement data.'' To avoid these complications, gradings were not fully treated in \cite{ManionRouquier}; correspondingly, when decategorifying in this paper, we will work with Grothendieck groups defined over $\F_2$ rather than over $\Z$, and we will view $\A(\Zc)$ as a differential algebra.
\end{remark}

\begin{definition}
We let $\A(\Zc,k)$ be the $\F_2$-subspace of $\A(\Zc)$ spanned by strands pictures such that the number of solid strands plus half the number of dotted strands is $k$. In fact, $\A(\Zc,k)$ is a dg subalgebra of $\A(\Zc)$ (ignoring unit), and if $|\mathbf{a}| = 2n$, we have $\A(\Zc) = \bigoplus_{k=0}^{n} \A(\Zc,k)$.
\end{definition}

The basis elements of $\A(\Zc)$ with only dotted (horizontal) strands are idempotents of $\A(\Zc)$. Furthermore, for a general basis element $a$ of $\A(\Zc)$, there is exactly one such idempotent (call it $\lambda(a)$) such that $\lambda(a) a = a$, and for all other such idempotents $\lambda'$, we have $\lambda' a = 0$. We will refer to $\lambda(a)$ as the left idempotent of $a$; we can define a right idempotent $\rho(a)$ similarly.

Below we will identify $\A(\Zc)$ with the differential category whose objects are in bijection with the all-horizontal basis elements of $\A(\Zc)$, and whose morphism space from $e$ to $e'$ is $e'\A(\Zc)e$. Because each basis element of $\A(\Zc)$ has a unique left and right idempotent, we can view these elements as giving a basis for the morphism spaces of $\A(\Zc)$ as a category.

\subsection{Higher actions on strands algebras}\label{sec:HigherActions}

Let $\Zc = (\mathbf{Z}, \mathbf{a}, M)$ be an arc diagram; as in \cite[Section 7.2.4]{ManionRouquier}, we can view $\Zc$ as a singular curve $Z$ in the language of that paper, and $\A(\Zc)$ is the endomorphism algebra of a collection of objects in the strands category $\Sc(Z)$ (see \cite[Section 7.4.11]{ManionRouquier}). For an interval $I$ in $\mathbf{Z}$ (equivalently, a non-circular component of $Z$ as in \cite[Section 7.2.2]{ManionRouquier}), the constructions of \cite[Section 8.1.1]{ManionRouquier} give us a differential bimodule $E$ over $\A(\Zc)$ (we will call this bimodule $\Ec$ for notational clarity). Closely related constructions appear in \cite{DM}, although in that paper the relevant pictures were not explicitly organized into a bimodule over $\A(\Zc)$.

As with the strands algebras, the bimodule $\Ec$ is defined by specifying an $\F_2$-basis of strands pictures, together with a differential and left and right actions of $\A(\Zc)$ in terms of basis elements. These strands pictures are almost the same as those described in Definition~\ref{def:StrandsPicture}. To describe the difference, let $P$ be the endpoint of the interval $I$ such that in the orientation on $\mathbf{Z}$, $I$ points from $P$ to its other endpoint. Then, in a strands picture for $\Ec$, there should be one solid strand with its left endpoint at $(1/2, P) \in [0,1] \times \mathbf{Z}$ and with its right endpoint in $\{1\} \times \mathbf{a}$. See Figure~\ref{fig:EStrandsPicture}; all other rules in Definition~\ref{def:StrandsPicture} are unchanged.

\begin{figure}
    \centering
    \includegraphics[scale=0.7]{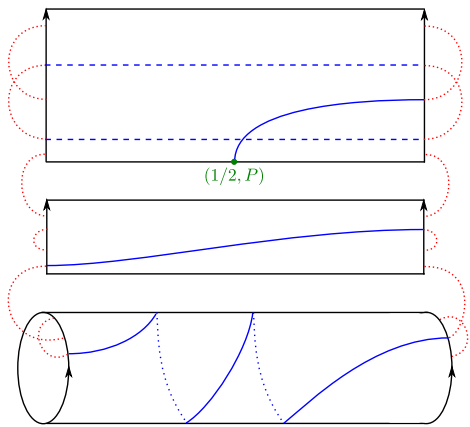}
    \caption{A strands picture for $\Ec$ (the distinguished interval $I$ is the top interval).}
    \label{fig:EStrandsPicture}
\end{figure}

\begin{definition}
As an $\F_2$-vector space, $\Ec$ is defined to be the formal span of the strands pictures described above, which form an $\F_2$-basis for $\Ec$. The left and right actions of $\A(\Zc)$ on $\Ec$, and the differential on $\Ec$, are defined by concatenation and resolution of crossings as in Definition~\ref{def:StrandsAlgebra}. We let $\Ec(k)$ be the $\F_2$-subspace of $\Ec$ spanned by strands pictures such that the number of solid strands plus half the number of dotted strands is $k$; then $\Ec(k)$ is a differential sub-bimodule of $\Ec$, and if $|\mathbf{a}| = 2n$, we have $\Ec = \bigoplus_{k=1}^n \Ec(k)$. Furthermore, $\Ec$ is a bimodule over $(\A(\Zc,k-1), \A(\Zc,k))$ with all other summands of $\A(\Zc)$ acting as zero on $\Ec$.
\end{definition}

\begin{figure}
    \centering
    \includegraphics[scale=0.7]{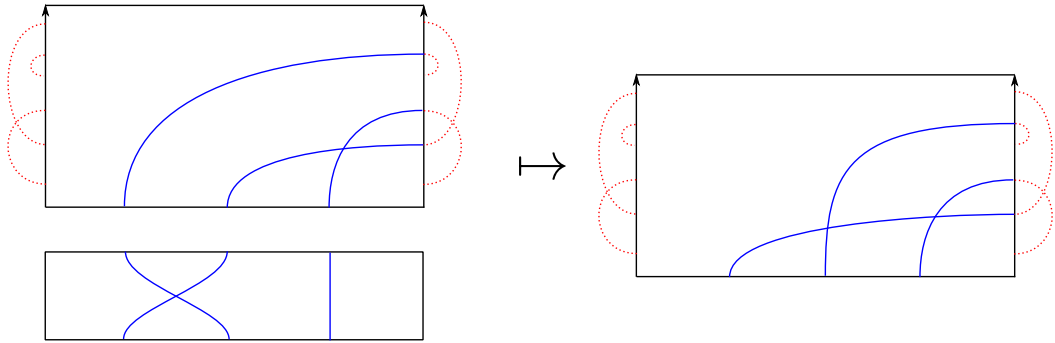}
    \caption{The action of an element of $NC_3$ on $\Ec^{\otimes 3}$.}
    \label{fig:NCAction}
\end{figure}

As with the basis elements of $\A(\Zc)$, to each basis element $x$ of $\Ec$ we can associate a left idempotent $\lambda(x)$ and a right idempotent $\rho(x)$. We have $x = \lambda(x) x \rho(x)$, while for any other purely-horizontal basis elements $\lambda' \neq \lambda(x)$, $\rho' \neq \rho(x)$ of $\A(\Zc)$, we have $\lambda' x = 0$ and $x \rho' = 0$.

By \cite[Lemma 8.1.2]{ManionRouquier}, the bimodule $\Ec \otimes_{\A(\Zc)} \Ec \otimes_{\A(\Zc)} \cdots \otimes_{\A(\Zc)} \Ec$ (with $m$ factors) is isomorphic to the bimodule defined analogously to $\Ec$, but having solid strands with left endpoints at $\{(\frac{1}{n+1}, P), (\frac{2}{n+1}, P), \ldots, (\frac{n}{n+1}, P)\}$. This bimodule (which we will call $\Ec^{\otimes m}$) also appears in \cite{DM}, and as in that paper it admits a left action of $NC_n$ defined diagrammatically by sticking strands pictures for $NC_n$ on the bottom of strands pictures for $\Ec^{\otimes m}$ (see Figure~\ref{fig:NCAction}). These actions form a 2-action of $\Uc$ on $\A(\Zc)$ via differential bimodules and bimodule maps, which was defined in \cite[Proposition 8.1.3]{ManionRouquier}. In other words, they give a differential monoidal functor from $\Uc$ to the dg monoidal category of differential bimodules over $\A(\Zc)$ and chain complexes of bimodule maps between them.

\subsection{Decategorification}\label{sec:Decategorification}

\subsubsection{Decategorifying $\Uc$}

\begin{definition}
For a differential category $A$, we let $\overline{A}$ denote the smallest full dg subcategory of $\Modl{A}$ (left differential modules over $A$) containing $\Hom(e,-)$ for all objects $e$ of $A$ and closed under mapping cones and isomorphisms. If $A$ is a dg category, we let $\Modl{A}$ be the category of left dg modules instead, and require that $\overline{A}$ be closed under degree shifts. We let $H(A)$ denote the homotopy category of $A$, and we let $A^i$ denote the idempotent completion of $A$.
\end{definition}

\begin{remark}
In the language of bordered Heegaard Floer homology \cite{LOTBorderedOrig, LOTBimodules}, $\overline{A}$ is essentially the same as the differential category of finitely generated bounded type $D$ structures over $A$ (in this setting it is typical to view $A$ as a differential algebra with a distinguished set of idempotents rather than as a dg category).
\end{remark}

It is a well-known result (see \cite[Corollary 3.7]{KellerOnDGCats}) that if $A$ is a dg category, then $H(\overline{A})^i$ is equivalent to the full subcategory of the derived category $\Dc(A)$ (of left dg $A$-modules) on compact objects, i.e. the compact derived category of $A$.

We can view dg algebras such as $NC_n$ as dg categories with one object. Khovanov shows in \cite{KhovOneHalf} that the Grothendieck group of the compact derived category of $NC_n$ is zero for $n \geq 2$. For $n = 0$ and $n=1$, $NC_n$ is $\F_2$, so the Grothendieck group of its compact derived category is $\Z$ (Khovanov gets $\Z[q,q^{-1}]$ instead because he introduces an extra $q$-grading on $NC_n$ which is identically zero, but we will not use this grading). 

\begin{corollary}
The Grothendieck group $K_0(H(\overline{NC_n}))$ is also $\Z$ for $n \in \{0,1\}$ and is zero for $n \geq 2$, where $H(\overline{NC_n})$ is the homotopy category of $\overline{NC_n}$.
\end{corollary}

\begin{proof}
The inclusion of the triangulated category $H(\overline{NC_n})$ into its idempotent completion is a monomorphism by \cite[Corollary 2.3]{ThomasonTriangulatedSubcats}. In fact, by \cite[Theorem 2.1]{ThomasonTriangulatedSubcats}, $H(\overline{NC_n})$ is already idempotent complete.
\end{proof}

Since we will primarily work with Grothendieck groups over $\F_2$ here, we introduce the following definition.

\begin{definition}
Let $\Cc$ be a category equipped with a collection of distinguished triangles $X \to Y \to Z \rightsquigarrow$ as in a triangulated category (but we do not require $\Cc$ to be triangulated or even to have a shift functor). We let $K_0^{\F_2}(\Cc)$ be the $\F_2$-vector space with basis given by isomorphism classes of objects of $\Cc$ modulo relations $[X] + [Y] + [Z] = 0$ whenever there exists a distinguished triangle $X \to Y \to Z \rightsquigarrow$.
\end{definition}

For a triangulated category $\Cc$, the above definition agrees with $K_0(\Cc) \otimes \F_2$. We see that $K_0^{\F_2}(H(\overline{NC_n})$ is isomorphic to $\F_2$ for $n \in \{0,1\}$ and is zero otherwise.

Now, since $\Uc$ is a direct sum of $NC_n$ (as a one-object dg category) over all $n \geq 0$, we have $K_0^{\F_2}(H(\overline{\Uc})) \cong \F_2 \oplus \F_2$. For notational convenience, we let
\[
K_0^{\F_2}(\Uc) := K_0^{\F_2}(H(\overline{\Uc})).
\]
Taking the monoidal structure on $\Uc$ into account, we see that as an $\F_2$-algebra, we have
\[
K_0^{\F_2}(\Uc) \cong \F_2[E]/(E^2)
\]
(this is Khovanov's identification $K_0(H^-) \cong \Z[q,q^{-1},E_1]/(E_1^2)$ from \cite{KhovOneHalf}, adapted to our setting).

\subsubsection{Decategorifying the strands algebras}

As mentioned above, we will view the strands algebras $\A(\Zc)$ as differential categories with multiple (but finitely many) objects in bijection with the set of purely-horizontal strands pictures for $\Zc$. The homotopy category $H(\overline{\A(\Zc)})$ has a collection of distinguished triangles, namely those isomorphic to the image in the homotopy category of $X \xrightarrow{f} Y \to \Cone(f) \rightsquigarrow$ for some closed morphism $f\colon X \to Y$ in $\overline{\A(\Zc)}$.
\begin{proposition}[\cite{PetkovaDecat}]
For $\Zc = (\mathbf{Z}, \mathbf{a}, M)$ with $\mathbf{Z}$ a single interval, $K_0(H(\overline{\A(\Zc)}))$ is isomorphic to $\wedge^* H_1(F;\Z)$ where $F$ is the surface represented by $\Zc$. Specifically, for each $k$, $K_0(H(\overline{\A(\Zc, k)}))$ is isomorphic to $\wedge^k H_1(F;\Z)$.
\end{proposition}

It follows that $K_0^{\F_2}(H(\overline{\A(\Zc)}))$ is isomorphic to $\wedge^* H_1(F;\F_2)$, and in the $\F_2$ setting we do not need to consider Petkova's absolute $\Z/2\Z$ homological grading on $\A(\Zc)$.

\begin{remark}
Petkova views the surface $F$ associated to a one-interval arc diagram $\Zc$ as being closed, while we view it as having $S^1$ boundary with one $S_+$ interval and one $S_-$ interval. Letting $\overline{F}$ denote the closed surface and $F$ denote the surface with boundary, we have natural identifications $H_1(\overline{F}) \cong H_1(F) \cong H_1(F,S_+)$ (with either $\Z$ or $\F_2$ coefficients).
\end{remark}

Petkova's arguments readily generalize to show that for general $\Zc$ as defined above, $K_0^{\F_2}(H(\overline{\A(\Zc)}))$ has an $\F_2$-basis given by the set of objects of $\A(\Zc)$ as a dg category, i.e. by the purely-horizontal strands pictures for $\Zc$. 

\begin{proposition}
If $(F,S_+,S_-,\Lambda)$ is the sutured surface represented by a general arc diagram $\Zc$, then the vector space $\wedge^* H_1(F,S_+; \F_2)$ has a basis in bijection with purely-horizontal strands pictures for $\Zc$.
\end{proposition}

\begin{proof}
The construction of $F$ from $\Zc = (\mathbf{Z}, \mathbf{a}, \Lambda)$ starts by taking $\mathbf{Z} \times [0,1]$, a collection of rectangles and annuli, and gluing on some $2$-dimensional $1$-handles. For each pair of points $\{p,q\}$ of $\mathbf{a}$ matched by $M$, one glues on a $1$-handle with attaching zero-sphere $\{ (p,1), (q,1) \}$ compatibly with the orientation on $\mathbf{Z}$. The result is $F$; one sets $S_+ := \mathbf{Z} \times \{0\}$ and $\Lambda := (\partial \mathbf{Z}) \times \{0\}$, with the rest of the boundary of $F$ placed in $S_-$.

It follows that $F / S_+$ is homotopy equivalent to a wedge product of circles, one for each pair of points of $\mathbf{a}$, and these circles form a basis for $H_1(F,S_+; \F_2)$. A basis for $\wedge^* H_1(F,S_+;\F_2)$ is then given by all subsets of the set of these circles. For each such subset $X$, there is a corresponding purely-horizontal strands picture for $\Zc$; if a circle (corresponding to $\{p,q\}$ matched by $M$) is in $X$, one draws a pair of dotted horizontal strands at $p$ and $q$ in the strands picture. This correspondence is a bijection, proving the proposition.
\end{proof}

Let $K_0^{\F_2}(\A(\Zc)) := K_0^{\F_2}(H(\overline{\A(\Zc)}))$ and $K_0^{\F_2}(\A(\Zc,k)) := K_0^{\F_2}(H(\overline{\A(\Zc,k)}))$.

\begin{corollary}\label{cor:K0Homology}
We have natural identifications
\[
K_0^{\F_2}(\A(\Zc)) \cong \wedge^* H_1(F,S_+; \F_2) \quad \textrm{and} \quad K_0^{\F_2}(\A(\Zc,k)) \cong \wedge^k H_1(F,S_+; \F_2).
\]
\end{corollary}

\section{Actions on exterior powers of homology}\label{sec:HomologyActions}

Let $\Zc = (\mathbf{Z}, \mathbf{a}, M)$ be an arc diagram representing a sutured surface $(F,S_+,S_-,\Lambda)$ as in Figure~\ref{fig:SurfaceExample}, and let $I$ be an interval component of $S_+$ (equivalently, let $I$ be an interval component of $\mathbf{Z}$). The endomorphism $\Phi_I$ of $\wedge^* H_1(F,S_+;\F_2)$ defined in the introduction squares to zero and thus gives us an action of $\F_2[E]/(E^2)$ on $\wedge^* H_1(F,S_+;\F_2)$ in which $E$ acts by $\Phi_I$. In this section we identify this action with the action of $K_0^{\F_2}(\Uc)$ on $K_0^{\F_2}(\A(\Zc))$ coming from the 2-action of $\Uc$ on $\A(\Zc)$ described in Section~\ref{sec:HigherActions}.

\begin{figure}
    \centering
    \includegraphics[scale=0.7]{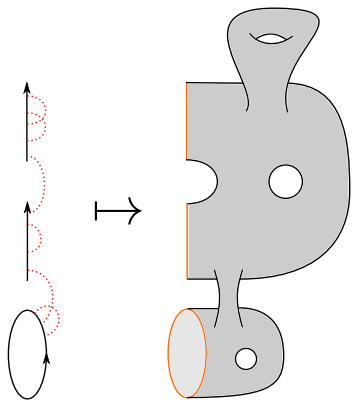}
    \caption{An arc diagram and the sutured surface it represents. The $S_+$ portion of the surface boundary is drawn in orange and the $S_-$ portion is drawn in black.}
    \label{fig:SurfaceExample}
\end{figure}

\begin{figure}
    \centering
    \includegraphics[scale=0.7]{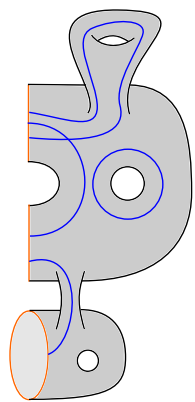}
    \caption{Depiction of a pure wedge-product element of $\wedge^* H_1(F,S_+;\F_2)$.}
    \label{fig:WedgeProduct}
\end{figure}

\begin{remark}
For an element $\omega$ of $\wedge^* H_1(F,S_+;\F_2)$ that is a pure wedge product of arcs in $F$ with boundary on $S_+$ and/or circles in $F$, we can depict $\omega$ by drawing all the arcs and circles of $\omega$ in a picture of $F$. See Figure~\ref{fig:WedgeProduct} for an example. The element $E$ of $\F_2[E]/(E^2)$ acts on this depiction of $\omega$ by summing over all ways of removing one arc incident with the component $I$ of $S_+$; see Figure~\ref{fig:EHomologyAction}. An arc with both endpoints on $I$ is ``removed twice'' which, in the sum with $\F_2$ coefficients, amounts to not being removed at all; indeed, such an arc represents the same homology class as a circle with no endpoints.
\end{remark}

\begin{figure}
    \centering
    \includegraphics[scale=0.7]{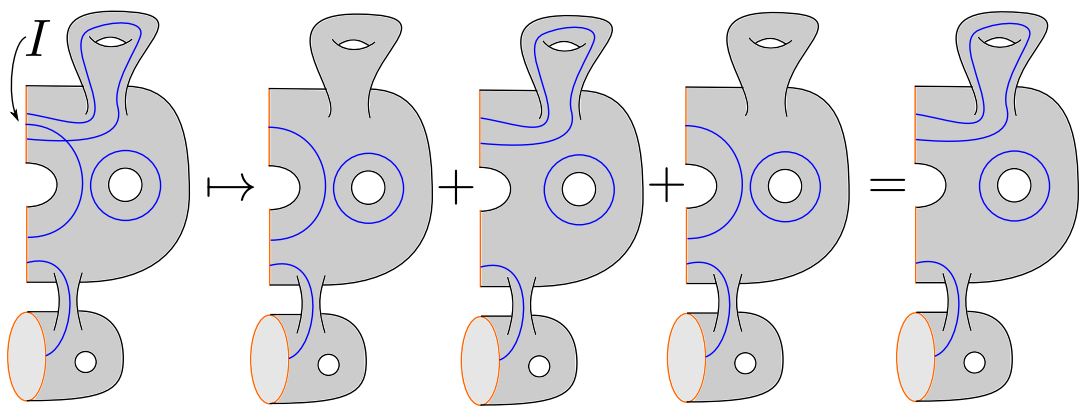}
    \caption{Action of $E \in \F_2[E]/(E^2)$ on $\omega \in \wedge^* H_1(F,S_+; \F_2)$ given a distinguished interval $I$ of $S_+$.}
    \label{fig:EHomologyAction}
\end{figure}

We first review an important structural property of the bimodule $\Ec$ from Section~\ref{sec:HigherActions}; the below proposition follows from \cite[Section 8.1.4]{ManionRouquier}, but to keep this paper self-contained we include an independent proof below. 

\begin{proposition}\label{prop:DA}
As a left differential module over the differential category $\A(\Zc)$, $\Ec$ is an object of $\overline{\A(\Zc)}$.
\end{proposition}

\begin{proof}
We first show that as a left module (disregarding the differential), $\Ec$ is isomorphic to a direct sum of modules of the form $\Hom(e,-)$ for objects $e$ of $\A(\Zc)$. Indeed, consider the subset $S$ of strands pictures for $\Ec$ (i.e. $\F_2$-basis elements of $\Ec$) such that the only moving strand is the one with left endpoint at $(1/2,P)$ in the language of Section~\ref{sec:HigherActions}. See Figure~\ref{fig:SElement} for an example of an element of $S$. An arbitrary basis element $x$ of $\Ec$ can be written as $a y$ for unique basis elements $a \in \A(\Zc)$ and $y \in S$; indeed, after a homotopy relative to the endpoints, we can draw $x$ such that all strands of $x$ except the one with endpoint at $(1/2,P)$ only move on $\mathbf{Z} \times [0,\varepsilon]$ for some $\varepsilon < 1/2$, and are horizontal on $\mathbf{Z} \times [\varepsilon,1]$ (see Figure~\ref{fig:StretchingStrandsPicture}).

Cutting the diagram for $x$ at $\mathbf{Z} \times \{\varepsilon\}$, we see a strands picture for a basis element $a \in \A(\Zc)$ on the left. On the right side of the cut, let $y$ be the element of $S$ obtained by making all the horizontal strands dotted and adding in their matching horizontal strands (according to the matching $M$). See Figure~\ref{fig:EStrandsFactorization} for an example. We have $a y = x$; furthermore, for any $y \in S$ with left idempotent $\lambda(y)$, and any basis element $a$ of $\Hom_{\A(\Zc)}(\lambda(y),-)$, we have that $a y$ is a basis element for $\Ec$ and that $a$ and $y$ are recovered when splitting $a y$ as above.

We have defined a bijection between our basis for $\Ec$ and the set of pairs $(a,y)$ where $y$ is an element of $S$ with left idempotent $\lambda(y)$ and $a$ is a basis element of $\Hom_{\A(\Zc)}(\lambda(y),-)$. Thus, we have an identification of $\Ec$ with $\bigoplus_{y \in S} \Hom_{\A(\Zc)}(\lambda(y),-)$ as vector spaces. This identification respects left multiplication by $\A(\Zc)$, so
\[
\Ec \cong \bigoplus_{y \in S} \Hom_{\A(\Zc)}(\lambda(y),-)
\]
as left modules over $\A(\Zc)$ (ignoring the differential).

Now, we can define a grading on the elements of $S$: say $y \in S$ has degree $d$ if the moving strand $\sigma$ of $y$ with left endpoint $(1/2,P)$ encounters $d$ points of $\mathbf{a}$ while traveling along a minimal path in $\mathbf{Z}$ from $P$ to its right endpoint. Order the elements of $S$ by increasing degree (choose any ordering of the elements of $S$ in each given degree). Because the differential on $\Ec$, applied to $y \in S$, will only resolve crossings between the special strand $\sigma$ of $y$ and horizontal strands strictly below $\sigma$, the only nonzero terms of this differential will be of the form $a y'$ for $y'$ of degree strictly less than that of $y$ (and thus $y'$ that appear before $y$ in the ordering on $S$). It follows that $\Ec$ is isomorphic to an iterated mapping cone built from $\Hom_{\A(\Zc)}(\lambda(y),-)$ for $y \in S$, so we have $\Ec \in \overline{\A(\Zc)}$.
\end{proof}

\begin{figure}
    \centering
    \includegraphics[scale=0.7]{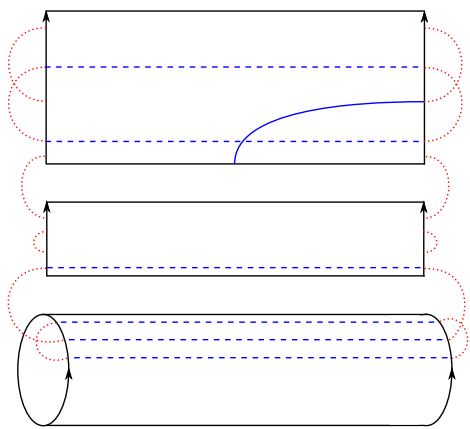}
    \caption{An element of the set $S$ of special basis elements of $\Ec$.}
    \label{fig:SElement}
\end{figure}

\begin{figure}
    \centering
    \includegraphics[scale=0.7]{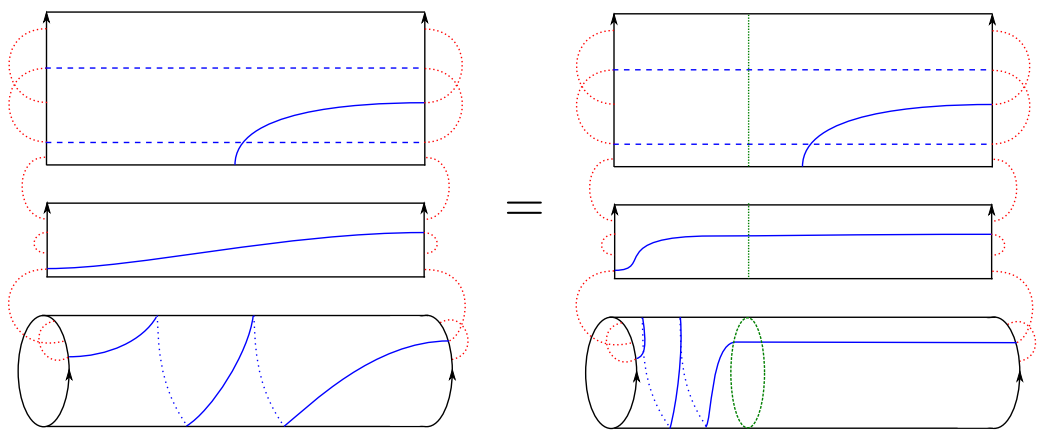}
    \caption{Stretching the basis element $x$ of Figure~\ref{fig:EStrandsPicture} so that all ``ordinary'' moving strands only move on $\mathbf{Z} \times [0,\varepsilon]$; the green dashed lines on the right indicate where we will cut to factor $x$ as $ay$ with $a \in \A(\Zc), y \in S$.}
    \label{fig:StretchingStrandsPicture}
\end{figure}

\begin{figure}
    \centering
    \includegraphics[scale=0.7]{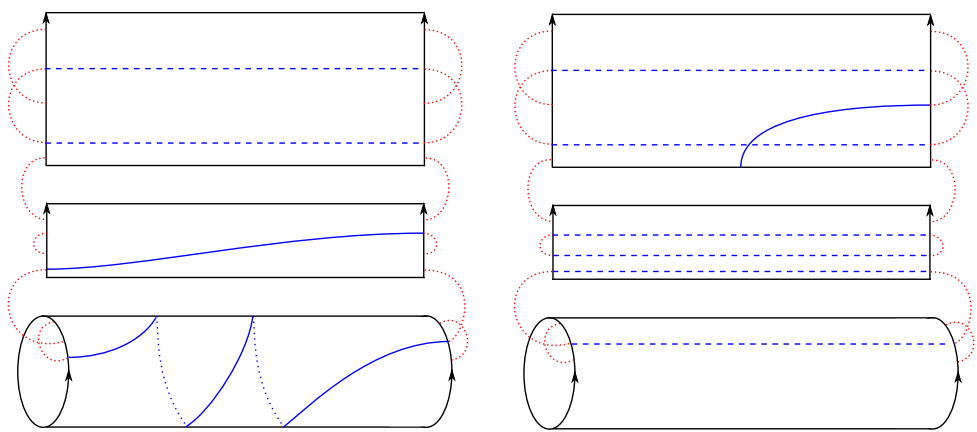}
    \caption{Factorizing the basis element $x$ of Figure~\ref{fig:EStrandsPicture} as $a \in \A(\Zc)$ (left) times $y \in S$ (right).}
    \label{fig:EStrandsFactorization}
\end{figure}

\begin{remark}
In the language of bordered Heegaard Floer homology, Proposition~\ref{prop:DA} says that $\Ec$ is the differential bimodule associated to a finitely generated left bounded type $DA$ bimodule over $\A(\Zc)$ with $\delta^1_i$ zero for $i > 2$.
\end{remark}

Proposition~\ref{prop:DA} implies the following corollary.

\begin{corollary}
We have a differential functor $\Ec \otimes_{\A(\Zc)} -$ from $\overline{\A(\Zc)}$ to itself, and thus a functor $\Ec \otimes_{\A(\Zc)} -$ from $H(\overline{\A(\Zc)})$ to itself.
\end{corollary}

The differential functor $\Ec \otimes_{\A(\Zc)} -$ sends mapping cones to mapping cones, so the corresponding functor on homotopy categories sends distinguished triangles to distinguished triangles and thus induces an endomorphism $[\Ec \otimes_{\A(\Zc)} -]$ of $K_0^{\F_2}(\A(\Zc))$.

\begin{theorem}\label{thm:MainDecatResult}
Let $\Zc = (\mathbf{Z},\mathbf{a},M)$ be an arc diagram and let $(F,S_+,S_-,\Lambda)$ be the sutured surface represented by $\Zc$. Let $I$ be an interval component of $S_+$, or equivalently an interval component of $\mathbf{Z}$. Under the identification $K_0^{\F_2}(\A(\Zc)) \cong \wedge^* H_1(F,S_+;\F_2)$ from Corollary~\ref{cor:K0Homology}, the endomorphism $[\Ec \otimes_{\A(\Zc)} -]$ of $K_0^{\F_2}(\A(\Zc))$ agrees with the endomorphism $\Phi_I$ of $\wedge^* H_1(F,S_+;\F_2)$ from the introduction. More specifically, the map $[\Ec(k) \otimes_{\A(\Zc,k)} -]$ from $K_0^{\F_2}(\A(\Zc),k)$ to $K_0^{\F_2}(\A(\Zc),k-1)$ agrees with $\Phi_I$ as a map from $\wedge^k H_1(F,S_+;\F_2)$ to $\wedge^{k-1} H_1(F,S_+;\F_2)$.
\end{theorem}

\begin{proof}
Let $e$ be an object of $\A(\Zc)$ (viewed as a differential category); we have a corresponding basis element $[\Hom(e,-)]$ of $K_0^{\F_2}(\A(\Zc))$. Applying $[\Ec \otimes_{\A(\Zc)} -]$ to $[\Hom(e,-)]$, we get $\sum_{y \in S, \, \rho(y) = e} [\Hom(\lambda(y), -)]$. Viewing $e$ as a purely horizontal strands picture and defining $S$ as in the proof of Proposition~\ref{prop:DA}, there is one element $y_s \in S$ with $\rho(y_s) = e$ for each strand $s$ of $e$ with endpoints in the interval $I$, and these are all the elements $y \in S$ with $\rho(y) = e$. For each such strand $s$ (say with endpoints at $Q \in I$), the element $y_s$ has a moving strand between $(1/2,P)$ and $(1,Q)$, and has the same horizontal strands as $e$ except for $s$ and its partner $s'$ under the matching. Thus, $\lambda(y_s)$ is $e$ with the strands $s$ and $s'$ removed. 

It follows that $[\Ec \otimes_{\A(\Zc)} -] ([\Hom(e,-)])$ is the sum of $[\Hom(e',-)]$ over all $e'$ obtained from $e$ by choosing one strand $s$ in $[0,1] \times I$ and removing both $s$ and its partner $s'$. In particular, for strands $s$ in $[0,1] \times I$ such that $s'$ is also in $[0,1] \times I$, the pair of strands $(s,s')$ is removed from $e$ twice, and since we are working over $\F_2$, removals of these strands contribute zero to $[\Ec \otimes_{\A(\Zc)} -] ([\Hom(e,-)])$.

Now let $\omega$ be the element of $\wedge^* H_1(F,S_+;\F_2)$ corresponding to $[\Hom(e,-)]$ under the isomorphism of Corollary~\ref{cor:K0Homology}. Concretely, each pair of matched strands $\{s,s'\}$ of $e$ gives a basis element of $H_1(F,S_+;\F_2)$, and $\omega$ is the wedge product of these elements over all such pairs $\{s,s'\}$. When we apply $\Phi_I$ to $\omega$, we sum over all ways to remove a factor from this wedge product if the factor maps to $1 \in \F_2$ under the map $\phi_I$ from the introduction. Such factors are those corresponding to pairs of strands $\{s,s'\}$ of $e$ in which one of $\{s,s'\}$, but not both, is in $[0,1] \times I$. It follows that $\Phi_I(\omega)$ corresponds to $[\Ec \otimes_{\A(\Zc)} -] ([\Hom(e,-)])$ as desired.
\end{proof}

\section{Gluing and TQFT}\label{sec:TQFT}

In this section, we prove (a slightly more general version of) Theorem~\ref{thm:IntroTQFT} from the introduction. Let $(F,S_+,S_-,\Lambda)$ be a sutured surface and suppose that $I_1 \neq I_2$ are interval components of $S_+$. Up to homeomorphism, there is a unique way to glue $I_1$ to $I_2$ and get an oriented surface $\overline{F}$. There are naturally defined subsets $\overline{S_+}$ and $\overline{S_-}$ of the boundary of $\overline{F}$, intersecting in a set of points $\overline{\Lambda}$ (which is $\Lambda$ with the endpoints of $I_1$ and $I_2$ removed).

\begin{figure}
    \centering
    \includegraphics[scale=0.7]{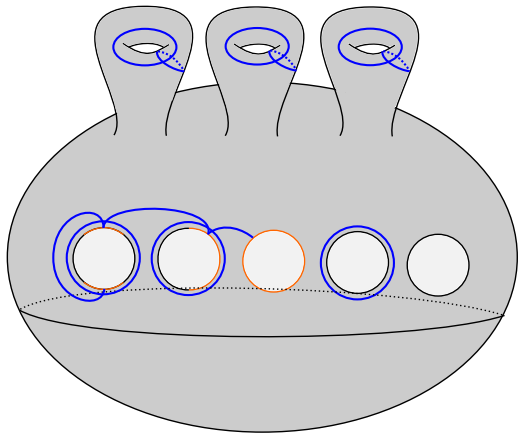}
    \caption{A standard model for a sutured surface, given by a sphere with some number of tori connect-summed on, as well as some number of disks removed and some even number of sutures on each boundary component. The $S_+$ boundary is drawn in orange and the $S_-$ boundary is drawn in black. The set of blue arcs and circles gives a basis for $H_1(F,S_+;\F_2)$.}
    \label{fig:StandardSuturedSurface}
\end{figure}

\begin{lemma}\label{lem:MainGluing}
We have an isomorphism
\[
\wedge^* H_1(\overline{F}, \overline{S_+}; \F_2) \cong \left( \wedge^* H_1(F,S_+;\F_2) \right) \otimes_{\left( \frac{\F_2[E]}{(E^2)} \right)^{\otimes 2}} \frac{\F_2[E]}{(E^2)},
\]
where the action of $\left( \F_2[E]/(E^2) \right)^{\otimes 2}$ on $\wedge^* H_1(F,S_+;\F_2)$ comes from the $\F_2[E]/(E^2)$ actions associated to $I_1$ and $I_2$, and the action of $\left( \F_2[E]/(E^2) \right)^{\otimes 2}$ on $\F_2[E]/(E^2)$ comes from multiplication. We can choose the isomorphism so that it intertwines the remaining actions of $\F_2[E]/(E^2)$ from $S_+$ intervals other than $I_1$ or $I_2$.
\end{lemma}

\begin{proof}
Pick a homeomorphism between $F$ and a finite disjoint union of standard sutured surfaces as shown in Figure~\ref{fig:StandardSuturedSurface} (spheres with some number of open disks removed and some even number of sutures on each boundary component, connect-summed with some number of tori). Figure~\ref{fig:StandardSuturedSurface} also indicates, with blue arcs and circles, a way to choose bases for $H_1(F,S_+;\F_2)$. One chooses:
\begin{itemize}
    \item For each torus that was connect-summed on, two circles giving a basis for the first homology of the torus;
    \item For all but one of the boundary components intersecting $S_-$ nontrivially, a circle around the boundary component;
    \item A continuous map from a connected acyclic graph $\Gamma_F$ to the surface $F$ (an embedding on each edge of $\Gamma_F$) with one vertex on each component of $S_+$. We will identify $\Gamma_F$ with its image in $F$.
\end{itemize}
These circles, together with the edges of $\Gamma_F$, give a basis for $H_1(F,S_+;\F_2)$, so subsets of this set of arcs and circles give a basis for $\wedge^* H_1(F,S_+;\F_2)$ consisting of wedge products of basis elements of $H_1(F,S_+;\F_2)$.

Now suppose $I_1$ and $I_2$ are intervals of $S_+$; we consider various cases. First, assume $I_1$ and $I_2$ live on distinct connected components of $F$. Choose $\Gamma_F$ such that the vertices on $I_1$ and $I_2$ (say $p_1$ and $p_2$) are leaves of $\Gamma_F$, i.e. they have degree 1. When gluing $F$ to get $\overline{F}$, we can ensure that $p_1$ and $p_2$ are glued to each other. If we let $e_1$ and $e_2$ denote the edges incident with $p_1$ and $p_2$, and modify $\Gamma_F$ by removing $p_1$, $p_2$, $e_1$ and $e_2$ while adding the edge $e_1 \cup e_2$ as an embedded arc in $\overline{F}$, we get an acyclic graph $\Gamma_{\overline{F}}$ embedded in $\overline{F}$ with one vertex on each component of $\overline{S_+}$. See Figure~\ref{fig:DifferentComponents} for an illustration.

Now, for an element $\omega$ of $\wedge^* H_1(F,S_+;\F_2)$ obtained as a wedge product of basis elements of $H_1(F,S_+;\F_2)$, the $I_1$-action of $E \in \F_2[E]/(E^2)$ on $\omega$ is zero if $e_1$ is not a wedge factor of $\omega$. Otherwise, write $\omega = e_1 \wedge \omega'$; we have $E \cdot \omega = \omega'$. 

The $I_2$-action of $E \in \F_2[E]/(E^2)$ on $\wedge^* H_1(F,S_+;\F_2)$ is similar; informally, $E$ acts by ``removing $e_2$.'' It follows that $\wedge^* H_1(F,S_+;\F_2)$ is a free module over $\left( \frac{\F_2[E]}{(E^2)} \right)^{\otimes 2}$ with an $\left( \frac{\F_2[E]}{(E^2)} \right)^{\otimes 2}$-basis given by elements $e_1 \wedge e_2 \wedge \omega'$ for all wedge products $\omega'$ in the other basis elements (not $e_1$ or $e_2$) of $H_1(F,S_+;\F_2)$. Thus, a basis for 
\[
\left( \wedge^* H_1(F,S_+;\F_2) \right) \otimes_{\left( \frac{\F_2[E]}{(E^2)} \right)^{\otimes 2}} \frac{\F_2[E]}{(E^2)}
\]
is given by the set of elements $e_1 \wedge e_2 \wedge \omega'$, together with the elements $e_1 \wedge \omega' = e_2 \wedge \omega'$ (in each case $\omega'$ is a wedge product of basis elements of $H_1(F,S_+;\F_2)$ that are not $e_1$ or $e_2$). Meanwhile, a basis for $\wedge^* H_1(\overline{F},\overline{S_+};\F_2)$ is given by the set of elements $(e_1 \cup e_2) \wedge \omega'$ and $\omega'$ for the same set of $\omega'$. We have a bijection between basis elements given by $e_1 \wedge e_2 \wedge \omega' \leftrightarrow (e_1 \cup e_2) \wedge \omega'$ and $(e_1 \wedge \omega' = e_2 \wedge \omega') \leftrightarrow \omega'$; this bijection is illustrated in Figure~\ref{fig:BasisBijection}. Thus, we have an isomorphism of vector spaces as claimed in the statement of the theorem. 

\begin{figure}
    \centering
    \includegraphics[scale=0.7]{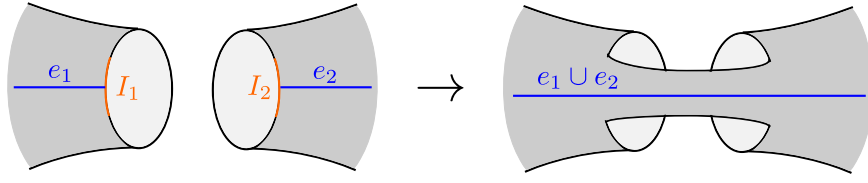}
    \caption{Left: arcs $e_1$ and $e_2$ in the surface $F$ before gluing. Right: the arc $e_1 \cup e_2$ after gluing $I_1$ to $I_2$.}
    \label{fig:DifferentComponents}
\end{figure}

\begin{figure}
    \centering
    \includegraphics{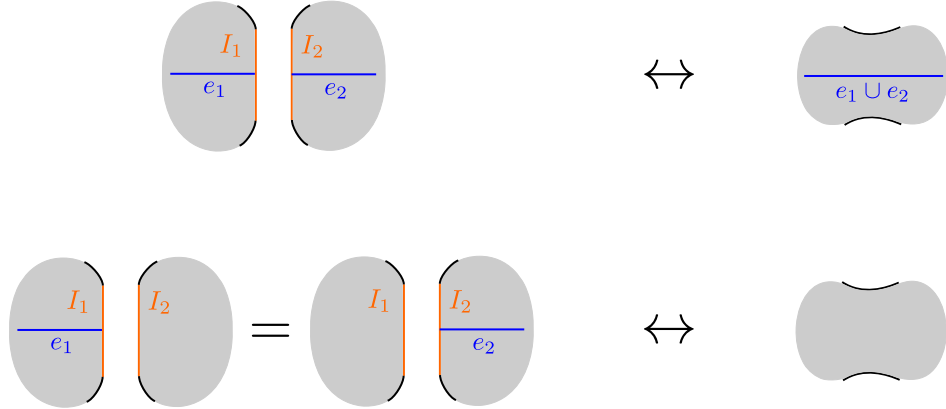}
    \caption{The bijection on basis elements in the first case of Lemma~\ref{lem:MainGluing}.}
    \label{fig:BasisBijection}
\end{figure}

To see that this isomorphism intertwines the remaining actions of $\F_2[E]/(E^2)$ for $S_+$ intervals that are not $I_1$ or $I_2$, it suffices to consider the actions for the other two intervals (say $I'_1$ and $I'_2$) that intersect $e_1$ and $e_2$ respectively. We will consider the action for $I'_1$; the case of $I'_2$ is similar. In the terminology used above, there are four types of basis elements of $\wedge^* H_1(F,S_+;\F_2)$: those of the forms $e_1 \wedge e_2 \wedge \omega'$, $e_1 \wedge \omega'$, $e_2 \wedge \omega'$, and $\omega'$. The $I'_1$-action of $E \in \F_2[E]/(E^2)$ sums over all ways to remove one wedge factor corresponding to an arc with exactly one endpoint on $I'_1$; besides terms that modify $\omega'$, there is a ``remove $e_1$'' term that sends $e_1 \wedge e_2 \wedge \omega'$ to $e_2 \wedge \omega'$ and sends $e_1 \wedge \omega'$ to $\omega'$. When we tensor over $\left( \F_2[E]/(E^2) \right)^{\otimes 2}$ with the identity map on $\F_2[E]/(E^2)$, the ``remove $e_1$'' term of the action of $E$ sends $e_1 \wedge e_2 \wedge \omega'$ to $e_2 \wedge \omega' = e_1 \wedge \omega'$ and sends $e_1 \wedge \omega' = e_2 \wedge \omega'$ to zero. On the other hand, as above there are two types of basis elements of $\wedge^* H_1(\overline{F}, \overline{S_+}; \F_2)$: those of the form $(e_1 \cup e_2) \wedge \omega'$ and those of the form $\omega'$. The $I'_1$-action of $E$ has terms modifying $\omega'$ in the same way as above, and it also has ``remove $e_1 \cup e_2$'' terms sending $(e_1 \cup e_2) \wedge \omega'$ to $\omega'$ and sending $\omega'$ to zero. It follows that our choice of isomorphism intertwines the $I'_1$ action of $\F_2[E]/(E^2)$.

Next, assume $I_1$ and $I_2$ live on the same connected component $F'$ of $F$; without loss of generality we can assume $F$ is connected so that $F' = F$. We consider two further cases: either $I_1$ and $I_2$ live on the same connected component of $\partial F$, or they live on different connected components of $\partial F$. First assume they live on the same component $C$ of $\partial F$, so that gluing $I_1$ to $I_2$ increases the number of boundary components of $F$ by one while keeping the genus the same. When choosing a basis for $H_1(F,S_+;\F_2)$ as above, we can choose $C$ for the unique not-fully-$S_+$ boundary component of $F$ that does not get a circle around it. We can also ensure that in the acyclic graph $\Gamma_F$, the vertices $p_1$ on $I_1$ and $p_2$ on $I_2$ are leaves of $\Gamma_F$.

\begin{figure}
    \centering
    \includegraphics{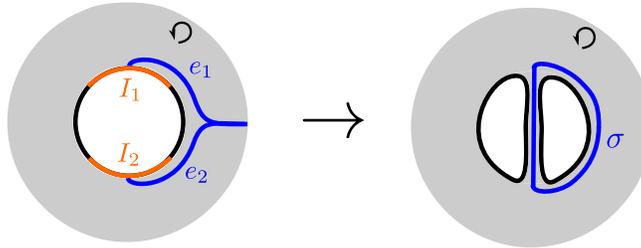}
    \caption{Left: local model near $C$ for the arcs $e_1$ and $e_2$. Right: the circle $\sigma$ after gluing $I_1$ to $I_2$. In both cases the curved arrow indicates the orientation on $F$; the induced boundary orientation on $C$ is clockwise in this figure.}
    \label{fig:NearCircle}
\end{figure}

If there are any intervals of $S_+$ other than $I_1$ and $I_2$, or any fully-$S_+$ circles, then $p_1$ and $p_2$ are incident with distinct edges $e_1 \neq e_2$ of $\Gamma_F$; we can furthermore choose $\Gamma_F$ so that $e_1$ and $e_2$ share an endpoint $q$, and such that as embedded submanifolds of $F$, they look like the left side of Figure~\ref{fig:NearCircle} in a small neighborhood of $C$ and are identical outside this neighborhood (the picture should be appropriately modified if $q$ lives on the circle $C$). As above, $\wedge^* H_1(F,S_+;\F_2)$ is free over $\left( \frac{\F_2[E]}{(E^2)} \right)^{\otimes 2}$ and has four types of basis elements, namely $e_1 \wedge e_2 \wedge \omega'$, $e_1 \wedge \omega'$, $e_2 \wedge \omega'$, and $\omega'$. A basis for 
\[
\left( \wedge^* H_1(F,S_+;\F_2) \right) \otimes_{\left( \frac{\F_2[E]}{(E^2)} \right)^{\otimes 2}} \frac{\F_2[E]}{(E^2)}
\] 
is given by the elements $e_1 \wedge e_2 \wedge \omega'$ along with the elements $e_1 \wedge \omega' = e_2 \wedge \omega'$. Meanwhile, we can take $\Gamma_{\overline{F}}$ to be $\Gamma_F$ with the edges $e_1$ and $e_2$ removed, and when choosing circles around boundary components to assemble a basis for $H_1(\overline{F}, \overline{S_+};\F_2)$, we can put a circle $\sigma$ around the component of $\partial \overline{F}$ containing the segment of $\partial F$ that goes from $I_1$ to $I_2$ when traversing the boundary in the oriented direction (see the right side of Figure~\ref{fig:NearCircle}). Then $\wedge^* H_1(\overline{F}, \overline{S_+}, \F_2)$ has basis elements of type $\sigma \wedge \omega'$ and $\omega'$; we identify these with elements of type $e_1 \wedge e_2 \wedge \omega'$ and $e_1 \wedge \omega' = e_2 \wedge \omega'$ respectively. This bijection on basis elements gives us an isomorphism of vector spaces as in the statement of the theorem.

To see that this isomorphism intertwines the remaining actions of $\F_2[E]/(E^2)$ from $S_+$ intervals other than $I_1$ or $I_2$, it suffices to consider the interval $I$ that contains the common endpoint $q$ of $e_1$ and $e_2$. The $I$-action of $E \in \F_2[E]/(E^2)$ on $\wedge^* H_1(F,S_+;\F_2)$ has terms that modify $\omega'$ as well as ``remove $e_1$'' terms sending (e.g.) $e_1 \wedge e_2 \wedge \omega'$ to $e_2 \wedge \omega'$ and ``remove $e_2$'' terms sending (e.g.) $e_1 \wedge e_2 \wedge \omega'$ to $e_1 \wedge \omega'$. When we tensor over $\left( \F_2[E]/(E^2) \right)^{\otimes 2}$ with the identity map on $\F_2[E]/(E^2)$, both the ``remove $e_1$'' and the ``remove $e_2$'' terms send $e_1 \wedge e_2 \wedge \omega'$ to $e_1 \wedge \omega' = e_2 \wedge \omega'$, and they send $e_1 \wedge \omega' = e_2 \wedge \omega'$ to zero. Since the ``remove $e_1$'' and ``remove $e_2$'' terms act in the same way, their contribution to the overall action of $E$ is zero, and only the ``modify $\omega'$'' terms remain. On the other hand, the $I$-action of $E$ on $\wedge^* H_1(F,S_+;\F_2)$ only modifies $\omega'$ in terms of type $\sigma \wedge \omega'$ or $\omega'$, since $\sigma$ is closed. It follows that our choice of isomorphism intertwines the $I$-action of $\F_2[E]/(E^2)$.

\begin{figure}
    \centering
    \includegraphics{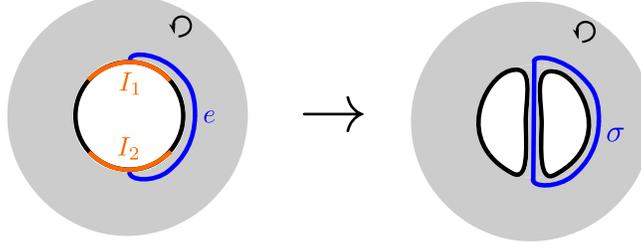}
    \caption{Left: local model near $C$ for the arc $e$. Right: the circle $\sigma$ after gluing $I_1$ to $I_2$.}
    \label{fig:NearCircleSpecial}
\end{figure}

Now assume that $I_1$ and $I_2$ are the only intervals of $S_+$ (but they still live on the same component $C$ of $\partial F$) and that there are no fully-$S_+$ circles; it follows that $\Gamma_F$ has a unique edge $e$ and it connects $p_1$ to $p_2$. We can assume $e$ lives in a small neighborhood of $C$, and that in this neighborhood it looks like the left side of Figure~\ref{fig:NearCircleSpecial}. The $I_1$-action and $I_2$-action of $E \in \F_2[E]/(E^2)$ on $\wedge^* H_1(F,S_+;\F_2)$ agree; they both send $e \wedge \omega'$ to $\omega'$ and send $\omega'$ to zero. Thus
\[
\left( \wedge^* H_1(F,S_+;\F_2) \right) \otimes_{\left( \frac{\F_2[E]}{(E^2)} \right)^{\otimes 2}} \frac{\F_2[E]}{(E^2)}
\] 
is canonically isomorphic to $\wedge^* H_1(F,S_+;\F_2)$ where no tensor operation is performed. Meanwhile, we can take $\Gamma_{\overline{F}}$ to be empty, but in assembling a basis for $H_1(\overline{F}, \overline{S_+};\F_2)$, we again put a circle $\sigma$ around the component of $\partial \overline{F}$ containing the segment of $\partial F$ that goes from $I_1$ to $I_2$ when traversing the boundary in the oriented direction (see the right side of Figure~\ref{fig:NearCircleSpecial}). The correspondences $e \wedge \omega' \leftrightarrow \sigma \wedge \omega'$ and $\omega' \leftrightarrow \omega'$ give an isomorphism of vector spaces as in the statement of the theorem. There are no remaining $S_+$ intervals, so we do not need to check that this isomorphism intertwines any actions.

\begin{figure}
    \centering
    \includegraphics[scale=0.7]{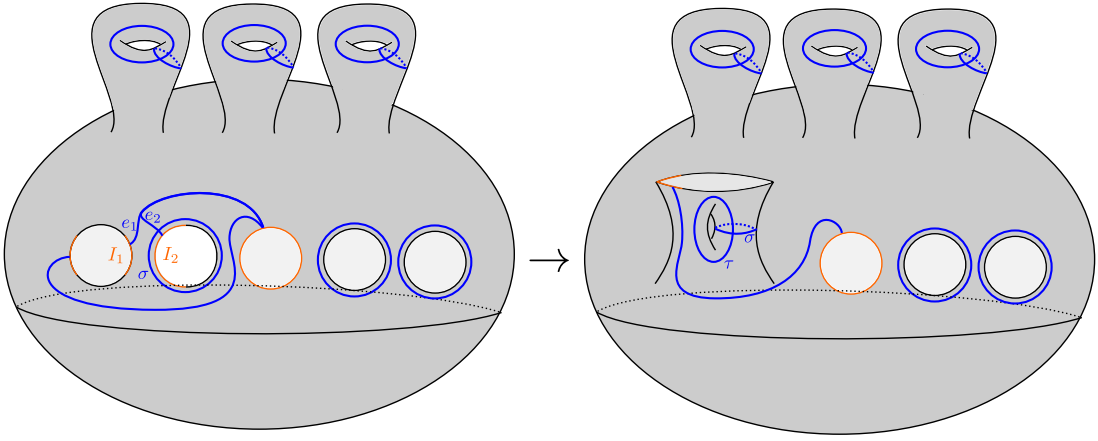}
    \caption{Left: $F$ before gluing intervals $I_1$, $I_2$ on the same component of $F$ but different components of $\partial F$. Right: the glued surface $\overline{F}$.}
    \label{fig:AddingGenus}
\end{figure}

Next, assume that $I_1$ and $I_2$ live on different components $C_1$ and $C_2$ of $\partial F$; for visual simplicity, assume that in the model for $F$ shown in Figure~\ref{fig:StandardSuturedSurface}, $C_1$ and $C_2$ are next to each other. Gluing $I_1$ to $I_2$ decreases the number of boundary components of $F$ by one and increases the genus of $F$ by one. Also assume that there is either at least one $S_+$ interval that is not $I_1$ or $I_2$, or that there is at least one fully-$S_+$ circle. As above, $p_1$ and $p_2$ are incident with distinct edges $e_1 \neq e_2$ of $\Gamma_F$, and we can choose $\Gamma_F$ so that $e_1$ and $e_2$ share a vertex $q$ and only diverge near $C_1$ and $C_2$. We also assume that $C_1$ is the unique not-fully-$S_+$ boundary circle of $F$ that does not get a circle around it as a basis element of $H_1(F,S_+;\F_2)$. Let $\sigma$ be the circle around $C_2$; see the left side of Figure~\ref{fig:AddingGenus}. 

Basis elements for $\wedge^* H_1(F,S_+;\F_2)$ can be of the form $e_1 \wedge e_2 \wedge \omega'$, $e_1 \wedge \omega'$, $e_2 \wedge \omega'$, or $\omega'$; when we tensor with $\F_2[E]/(E^2)$ over $\left( \F_2[E]/(E^2) \right)^{\otimes 2}$, we have a basis whose elements are of type $e_1 \wedge e_2 \wedge \omega'$ or $e_1 \wedge \omega' = e_2 \wedge \omega'$. Meanwhile, we choose a basis for $H_1(\overline{F}, \overline{S_+}, \F_2)$ by choosing a homeomorphism with the standard surface shown on the right side of Figure~\ref{fig:AddingGenus}. The graph $\Gamma_{\overline{F}}$ can be understood as $\Gamma_F$ with $e_1$ and $e_2$ removed; we also have basis elements $\sigma$ and $\tau$ of $H_1(F,S_+;\F_2)$ where $\sigma \subset \overline{F}$ comes from $\sigma \subset F$ and $\tau$ comes from $e_1$ and $e_2$. Basis elements of $\wedge^* H_1(\overline{F},\overline{S_+};\F_2)$ are of the form $\tau \wedge \omega'$ or $\omega'$ where $\omega'$ is a wedge product of basis elements for $H_1(\overline{F},\overline{S_+};\F_2)$ that are not $\tau$. The correspondence $e_1 \wedge e_2 \wedge \omega' \leftrightarrow \tau \wedge \omega'$ and $(e_1 \wedge \omega' = e_2 \wedge \omega') \leftrightarrow \omega'$ gives an isomorphism of vector spaces as in the statement of the theorem. The proof that this isomorphism intertwines the remaining actions of $\F_2[E]/(E^2)$ proceeds as above.

Finally, assume that $I_1$ and $I_2$ are the only $S_+$ intervals and that there are no fully-$S_+$ circles (while $I_1$ and $I_2$ still live on different components of $\partial F$). Letting $e$ be the arc of $\Gamma_F$ connecting $p_1 \in I_1$ to $p_2 \in I_2$, basis elements for $\wedge^* H_1(F,S_+;\F_2)$ are of the form $e \wedge \omega'$ or $\omega'$. Meanwhile, defining $\tau$ as in Figure~\ref{fig:AddingGenus}, basis elements for $\wedge^* H_1(\overline{F}, \overline{S_+};\F_2)$ are of the form $\tau \wedge \omega'$ or $\omega'$. The correspondence $e \wedge \omega' \leftrightarrow \tau \wedge \omega'$ and $\omega' \wedge \omega'$ gives an isomorphism of vector spaces as in the statement of the theorem, and there are no remaining actions for this isomorphism to intertwine.
\end{proof}

Lemma~\ref{lem:MainGluing} implies the following theorem.

\begin{theorem}\label{thm:TQFTGluing}
Let $(F,S_+,S_-,\Lambda)$ and $(F',S'_+, S'_-,\Lambda')$ be two sutured surfaces. For some $m \geq 0$, choose distinct intervals $I_1, \ldots, I_m$ of $S_+$ and distinct intervals $I'_1, \ldots, I'_m$ of $S'_+$. Use $I_1, \ldots, I_m$ to define an action of $\left( \F_2[E]/(E^2) \right)^{\otimes m}$ on $\wedge^* H_1(F,S_+;\F_2)$, and similarly for $F'$. Let $(\overline{F}, \overline{S_+}, \overline{S_-},\overline{\Lambda})$ be the sutured surface obtained by gluing $I_j$ to $I'_j$ for $1 \leq j \leq m$ (in such a way that the result is oriented). Then we have an isomorphism
\[
\wedge^* H_1(\overline{F}, \overline{S_+};\F_2) \cong \wedge^* H_1(F, S_+;\F_2) \otimes_{\left( \F_2[E]/(E^2) \right)^{\otimes m}} \wedge^* H_1(F',S'_+;\F_2)
\]
that intertwines the remaining actions of $\F_2[E]/(E^2)$ for intervals of $S_+$ and $S'_+$ that are not included in $\{ I_1,\ldots,I_m\}$ or $\{I'_1,\ldots,I'_m\}$.
\end{theorem}

\begin{proof}
We can write
\[
\wedge^* H_1(F, S_+;\F_2) \otimes_{\left( \F_2[E]/(E^2) \right)^{\otimes m}} \wedge^* H_1(F',S'_+;\F_2)
\]
as
\[
\left(\left( \wedge^* H_1(F \sqcup F', S_+ \sqcup S'_+; \F_2) \right) \otimes_{\left(\F_2[E]/(E^2) \right)^{\otimes 2}} \F_2[E]/(E^2) \right) \ldots \otimes_{\left(\F_2[E]/(E^2) \right)^{\otimes 2}} \F_2[E]/(E^2).
\]
where there are $m$ successive tensor products by $\F_2[E]/(E^2)$ over $\left(\F_2[E]/(E^2) \right)^{\otimes 2}$ (one for each pair $(I_j, I'_j)$). The result now follows from Lemma~\ref{lem:MainGluing}.
\end{proof}

\begin{corollary}
There is a functor from the full subcategory of the $1+1$-dimensional oriented open-closed cobordism category on objects with no closed circles (the ``open sector'' of the open-closed cobordism category) to ($\F_2$-algebras, bimodules up to isomorphism) sending an object with $m$ intervals to $\left( \F_2[E]/(E^2) \right)^{\otimes m}$ and sending a morphism (viewed as a sutured surface $(F,S_+,S_-,\Lambda)$) to $\wedge^* H_1(F,S_+;\F_2)$ (viewed as a bimodule over tensor products of $\F_2[E]/(E^2)$ for the input and output intervals of the morphism).
\end{corollary}

\section{The tensor product case}\label{sec:Tensor}

Figure~\ref{fig:TensorProduct} shows the open pair of pants surface $P$ with a sutured structure $(P,S_+,S_-,\Lambda)$. Let $e_1$ and $e_2$ be the arcs shown in the figure and let $I_1$, $I_2$, and $I_3$ be the $S_+$ intervals shown in the figure. Since $\{e_1,e_2\}$ is a basis for $H_1(P,S_+;\F_2)$, we have a basis $\{1, e_1, e_2, e_1 \wedge e_2\}$ for $\wedge^* H_1(P,S_+;\F_2)$. The three actions of $\F_2[E]/(E^2)$ on $\wedge^* H_1(P,S_+;\F_2)$ can be described as follows:
\begin{itemize}
    \item For the $I_1$-action, $E$ sends $1 \mapsto 0$, $e_1 \mapsto 1$, $e_2 \mapsto 0$, and $e_1 \wedge e_2 \mapsto e_2$.
    \item For the $I_2$-action, $E$ sends $1 \mapsto 0$, $e_1 \mapsto 0$, $e_2 \mapsto 1$, and $e_1 \wedge e_2 \mapsto e_1$.
    \item For the $I_3$-action, $E$ sends $1 \mapsto 0$, $e_1 \mapsto 1$, $e_2 \mapsto 1$, and $e_1 \wedge e_2 \mapsto e_1 + e_2$.
\end{itemize}
Using the $I_1$ and $I_2$ actions to define an action of $\left( \F_2[E]/(E^2) \right)^{\otimes 2}$ on $\wedge^* H_1(P,S_+;\F_2)$, we see that $\wedge^* H_1(P,S_+;\F_2)$ is a free module of rank $1$ over $\left( \F_2[E]/(E^2) \right)^{\otimes 2}$ with an $\left( \F_2[E]/(E^2) \right)^{\otimes 2}$-basis given by $\{e_1 \wedge e_2\}$. The $I_3$-action of $\F_2[E]/(E^2)$ is then given by applying the coproduct $\Delta(E) = E \otimes 1 + 1 \otimes E$, followed by multiplication in $\left( \F_2[E]/(E^2) \right)^{\otimes 2}$.

Now, if we have sutured surfaces $(F',S'_+,S'_-,\Lambda')$ and $(F'',S''_+,S''_-,\Lambda'')$ with chosen intervals $I'$ and $I''$ in $S'_+$ and $S''_+$ respectively, we can glue $F' \sqcup F''$ to $P$ by gluing $I'$ to $I_1$ and $I''$ to $I_2$. Applying Theorem~\ref{thm:TQFTGluing} with $F_1 := F' \sqcup F''$ and $F_2 := P$, and letting $(\overline{F}, \overline{S_+}, \overline{S_-}, \overline{\Lambda})$ denote the glued surface, we have
\begin{align*}
\wedge^* H_1(\overline{F}, \overline{S_+};\F_2) &\cong \left( \F_2[E]/(E^2) \right)^{\otimes 2} \otimes_{\left( \F_2[E]/(E^2) \right)^{\otimes 2}} \wedge^* H_1(F' \sqcup F'', S'_+ \sqcup S''_+; \F_2) \\
&\cong \wedge^* H_1(F' \sqcup F'', S'_+ \sqcup S''_+; \F_2) \\
&\cong \wedge^* H_1(F', S'_+;\F_2) \otimes \wedge^* H_1(F'',S''_+;\F_2)
\end{align*}
with $I_3$-action of $E$ given by taking $\Delta(E) = E \otimes 1 + 1 \otimes E$ and then acting on the tensor product $\wedge^* H_1(F', S'_+;\F_2) \otimes \wedge^* H_1(F'',S''_+;\F_2)$. Corollary~\ref{cor:IntroTensor} follows from this computation.

\begin{figure}
    \centering
    \includegraphics[scale=0.7]{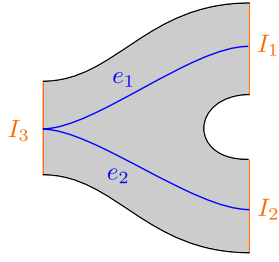}
    \caption{The open pair-of-pants surface $P$ with sutured structure and basis $\{e_1,e_2\}$ for $H_1(P,S_+;\F_2)$.}
    \label{fig:TensorProduct}
\end{figure}

\bibliographystyle{alpha}
\bibliography{biblio.bib}

\end{document}